\documentclass{elsarticle}
 \usepackage{graphicx,amssymb,amstext,amsmath}
 \usepackage{bigstrut}
 \usepackage{enumerate}
 \usepackage{amsthm}
 \usepackage[english]{babel}
 \usepackage[all]{xy}
\usepackage{framed}
\usepackage{caption}

 \newtheorem{theorem}{Theorem}[section]
 \newtheorem{thm}[theorem]{Theorem}

 \newtheorem{prop}[theorem]{Proposition}

\makeatletter
\def\@author#1{\g@addto@macro\elsauthors{\normalsize%
    \def\baselinestretch{1}%
    \upshape\authorsep#1\unskip\textsuperscript{%
      \ifx\@fnmark\@empty\else\unskip\sep\@fnmark\let\sep=,\fi
      \ifx\@corref\@empty\else\unskip\sep\@corref\let\sep=,\fi
      }%
    \def\authorsep{\unskip,\space}%
    \global\let\@fnmark\@empty
    \global\let\@corref\@empty  
    \global\let\sep\@empty}%
    \@eadauthor={#1}
}
\makeatother

\begin{document}

\journal{Discrete Mathematics}
\title{A note on  perfect quantum state transfers on trees}

\author{Bahman Ahmadi\corref{cor1}}
\ead{bahman.ahmadi@shirazu.ac.ir}
 \cortext[cor1]{Corresponding author}
 \author{Ahmad Mokhtar}
\ead{ahmad99mokhtar@gmail.com}
\address{Department of Mathematics, Shiraz University, Shiraz, Iran}

 \begin{abstract}
It has been asked in \cite{godsil2012state} whether there are trees other than $P_2$ and $P_3$ which can admit perfect state transfers. In this note we show that the answer is negative. 
 \end{abstract}


\maketitle

\section{\bf Introduction}\label{introduction}
For any simple graph $X$ with the adjacency matrix $A$ and with $|V(X)|=n$, we define the function $H_X(t)$ as
\[H_X(t)=\exp{(iAt)},\quad \text{for any } t.\]
If $X$ is clear from the context, we may just write $H(t)$. We say there is a \textit{perfect state transfer} or a PST between distinct vertices $u$ and $v$ of $X$ at time $\tau$, if  $|H(\tau)_{u,v}|=1$. For the motivation of this definition in designing quantum communication networks and a survey of important results, the reader may refer to \cite{christandl2004perfect,godsil2010can, godsil2012state}.

Godsil provides a proof in \cite{godsil2012state} that there is a PST between the endpoints of the paths $P_2$ and $P_3$. Also, the following has been proved in  \cite{christandl2004perfect}.

\begin{prop}\label{P_n}
The path $P_n$ has no PST for any $n\geq 4$. \qed
\end{prop}

Therefore, Godsil asks in \cite{godsil2012state}  whether there are any trees besides $P_2$ and $P_3$ on which a PST can occur. We prove that the answer is no. The main tool to do this is the following result also from \cite{godsil2012state}. Given any vertex $u$ from a graph $X$, we denote by $\Delta_u$ the distance partition of the vertices of $X$ with respect to $u$.

\begin{prop}\label{partition}
Let $u$ and $v$ be vertices in $X$. If there is perfect state transfer from $u$ to $v$, then $\Delta_u = \Delta_v$.\qed
\end{prop}

\section{There is no PST on trees}\label{trees}
In this section we prove the main result of the note.

\begin{thm}\label{main}
If $T\neq P_2, P_3$ is a tree, then there is  no PST on $T$.
\end{thm}
\begin{proof}
Suppose that there is a PST on $T$ between two distinct vertices $u$ and $v$. Assume $P: u=w_0-w_1-w_2-\cdots- w_r - w_{r+1}=v$ is the unique path between $u$ and $v$. First we show that both $u$ and $v$ must be leaves. If $u$ is adjacent to a vertex $z\neq w_1$, then $w_1$ and $z$ belong to the same cell of the distance partition $\Delta_u$, while they are in distinct cells of the partition $\Delta_v$. Therefore $\Delta_u\neq \Delta_v$ which, according to Proposition~\ref{partition}, is a contradiction. Hence $u$ is a leaf and with the same argument, $v$ is a leaf as well. Then we show that indeed $T=P$. To do this, suppose (for contrary)  that there is an $i\in \{1,\ldots,r\}$ such that $w_i$ has a neighbour $z$ other than $w_{i-1}$ and $w_{i+1}$. Then $w_{i-1}$ and $z$ belong to the same cell of  $\Delta_v$ while they belong to distinct cells of $\Delta_u$. This, similarly, is a contradiction. Thus the claim is proved; that is, $T$ is a path and since $T\neq P_2, P_3$, according to Proposition~\ref{P_n}, $T$ cannot have a PST.
\end{proof}

\section*{Acknowledgment}
The first author would like to thank the financial support the Iranian National Elites' Foundation.

\end{document}